\documentclass[12pt]{article}

\usepackage{mathrsfs}
\usepackage{color}
\usepackage{amssymb, amsthm, amsfonts, amsxtra, amsmath}

\usepackage{latexsym}
\usepackage[all]{xy}
\usepackage{graphics}
\usepackage{latexsym}
\usepackage{makeidx}
\usepackage{rotating}

\theoremstyle{change}

\newtheorem{Theorem}{Theorem}[section]
\newtheorem{Def}[Theorem]{Definition}
\newtheorem{Lem}[Theorem]{Lemma}
\newtheorem{Prop}[Theorem]{Proposition}
\newtheorem{Cor}[Theorem]{Corollary}
\newtheorem{Def-Prop}[Theorem]{Definition-Proposition}

\date{}

\begin{document}

\title{On cocycle twisting of compact quantum groups}
\author{Kenny De Commer\footnote{Research Assistant of the Research Foundation - Flanders (FWO -
Vlaanderen).}\\ \small Department of Mathematics,  K.U. Leuven\\
\small Celestijnenlaan 200B, 3001 Leuven, Belgium\\ \\ \small e-mail: kenny.decommer@wis.kuleuven.be}
\maketitle

\newcommand{\acnabla}{\nabla\!\!\!{^\shortmid}}
\newcommand{\undersetmin}[2]{{#1}\underset{\textrm{min}}{\otimes}{#2}}
\newcommand{\otimesud}[2]{\overset{#2}{\underset{#1}{\otimes}}}

\newcommand{\otimesmin}{\underset{\textrm{min}}{\otimes}}
\newcommand{\bigback}{\!\!\!\!\!\!\!\!\!\!\!\!\!\!\!\!\!\!\!\!\!\!\!\!}

\abstract{\noindent In this article, we give a class of examples of compact quantum groups and unitary 2-cocycles on them, such that the twisted quantum groups are non-compact, but still locally compact quantum groups (in the sense of Kustermans and Vaes). This also provides examples of cocycle twists where the underlying C$^*$-algebra of the quantum group changes.}

\section*{Introduction}

\noindent In the seventies, Kac and Vainerman (\cite{Kac1}), and independently Enock and Schwartz (\cite{Eno1}), introduced the notion of (what was called by the latter) a \emph{Kac algebra}, based on the fundamental work of Kac concerning ring groups in the sixties (\cite{Kac2}). Such Kac algebras, which are von Neumann algebras $M$ with a coproduct $\Delta: M\rightarrow M\otimes M$ satisfying certain conditions, can naturally be made into a category, containing as a full sub-category the category of all locally compact groups, but allowing a duality functor which extends the Pontryagin duality functor on the sub-category of all \emph{abelian} locally compact groups. Moreover, Kac algebras with a \emph{commutative} underlying von Neumann algebra arise precisely from locally compact groups, by passing to the $\mathscr{L}^{\infty}$-space of the latter with respect to the left (or right) Haar measure, and with $\Delta$ dual to the multiplication in the group.\\

\noindent However, these Kac algebras do not cover all interesting examples of what could be called `locally compact quantum groups'. In the eighties, Woronowicz introduced `compact matrix quantum groups' (\cite{Wor1}), which are to be seen as quantum versions of compact Lie groups. He also constructed in that paper certain compact matrix quantum groups $SU_q(2)$, which are deformations of the classical $SU(2)$-group by some parameter $q\in \mathbb{R}$ with $0<|q|< 1$. These quantum groups do not fit into the Kac algebra framework. The reason for this is that the antipode of these quantum groups is no longer a $^*$-preserving anti-automorphism, but some unbounded operator on the associated C$^*$-algebra of the quantum group.\\

\noindent A satisfactory theory, covering both the compact quantum groups, the Kac algebras and some isolated examples, was obtained in 2000, when Kustermans and Vaes introduced their C$^*$-algebraic quantum groups (\cite{Kus1}). In a follow-up paper, they also introduced von Neumann algebraic quantum groups (\cite{Kus2}), and proved that the C$^*$-algebra approach and the von Neumann algebra approach were just different ways to look at the same structure (in that one can pass from the von Neumann algebra setup to a (reduced or universal) C$^*$-algebraic setup and back). We also remark that in \cite{VDae2}, an slightly alternative approach to von Neumann algebraic quantum groups was presented. In this paper, we will be mainly using the von Neumann algebraic approach, which simply asks for the existence of a coproduct and invariant weights on a von Neumann algebra (see Definition \ref{DefvNeuQG}).\\

\noindent An interesting and important part of the theory consists in finding construction methods for von Neumann algebraic quantum groups. For example, in \cite{Baa1} the double product construction was worked out, while in \cite{Vae3} the bicrossed product construction was treated. In \cite{DeC2}, we developed another construction method, namely the generalized twisting of a von Neumann algebraic quantum group (by a Galois object for its dual). This covers in particular the twisting by unitary 2-cocycles, special situations of which had been considered by Enock and Vainerman in \cite{Eno2}, and by Fima and Vainerman in \cite{Fim1}.\\

\noindent When we have two von Neumann algebraic quantum groups, one of which is obtained from the other by the above generalized twisting construction, we call them \emph{comonoidally W$^*$-Morita} equivalent. The reason for this name is simple: the underlying von Neumann algebras of two such quantum groups are Morita equivalent (in the sense of Rieffel (\cite{Rie2})), with the equivalence `respecting the coproduct structure'. One can also show then that the representation categories of their associated universal C$^*$-algebraic quantum groups are unitarily comonoidally equivalent (so that they have the `same' \emph{tensor} category of $^*$-representations). The special case of cocycle twisting corresponds to those comonoidal Morita equivalences whose underlying Morita equivalence is (isomorphic to) the identity.\\

\noindent It was shown in \cite{DeC3} that comonoidal W$^*$-Morita equivalence provides a genuine equivalence relation between von Neumann algebraic quantum groups. It is then a natural question to find out which properties are preserved by this equivalence relation. In \cite{BDV1}, it was shown for example that the discreteness of a quantum group is preserved, while its amenability is not. It also follows from the results of that paper that `being discrete and Kac' is \emph{not} preserved. In the general setting of von Neumann algebraic quantum groups, we showed in \cite{DeC3} that the scaling constant is an invariant for comonoidal W$^*$-Morita equivalence. In this article, we will show in a very concrete way that the notion of `being compact' is \emph{not} an invariant. This implies in particular that the representation category of a locally compact quantum group, as a monoidal W$^*$-category, does \emph{not necessarily remember the topology of the quantum group}. In \cite{DeC3}, we have also shown, by more general methods, that `being compact and Kac' \emph{is} an invariant.\\

\noindent This article is divided into two sections. \emph{In the first section}, we recall some notions concerning compact quantum groups, von Neumann algebraic quantum groups and the twisting by unitary 2-cocycles. \emph{In the second section}, we recall the definition of Woronowicz's $SU_q(2)$ quantum groups. We then consider an infinite tensor product of these compact quantum groups over varying $q$'s, with the condition that the $q$'s go to zero sufficiently fast. Taking a limit of appropriate coboundaries, we obtain a 2-cocycle $\Omega$ on this infinite product quantum group which is no longer a coboundary. By giving an explicit formula for a non-finite but semi-finite invariant weight, we conclude that the $\Omega$-twisted quantum group is no longer compact.\\

\noindent \textit{Remarks on notation:} When $A$ is a set, we denote by $\iota$ the identity map $A\rightarrow A$. \\

\noindent We will need the following tensor products: we denote by $\odot$ the algebraic tensor product of two vector spaces over $\mathbb{C}$, by $\otimesmin$ the minimal tensor product of two C$^*$-algebras, and by $\otimes$ the spatial tensor product between von Neumann algebras or Hilbert spaces.\\

\noindent We will further use the following notations concerning weights on von Neumann algebras. If $M$ is a von Neumann algebra, and $\varphi:M^+\rightarrow \lbrack 0,+\infty\rbrack$ an nsf (= normal semi-finite faithful) weight, we denote by $\mathscr{N}_{\varphi}\subseteq M$ the $\sigma$-weakly dense left ideal of square integrable elements: \[\mathscr{N}_{\varphi} = \{x\in M \mid \varphi(x^*x) <\infty\}.\] We denote by $\mathscr{L}^2(M,\varphi)$ the Hilbert space completion of $\mathscr{N}_{\varphi}$ with respect to the inner product \[\langle x,y\rangle = \varphi(y^*x),\] and by $\Lambda_{\varphi}$ the canonical embedding $\mathscr{N}_{\varphi}\rightarrow \mathscr{L}^2(M,\varphi)$. We remark that $\Lambda_{\varphi}$ is ($\sigma$-strong)-(norm) closed. We denote $\mathscr{M}_{\varphi}^+$ for the space of elements $x\in M^+$ for which $\varphi(x)<\infty$, and $\mathscr{M}_{\varphi}$ for the complex linear span of $\mathscr{M}_{\varphi}^+$. One can show then that $\mathscr{M}_{\varphi} = \mathscr{N}_{\varphi}^*\cdot \mathscr{N}_{\varphi}$, i.e.~any element $x$ of $\mathscr{M}_{\varphi}$ can be written as $\sum_{i=1}^n x_i^*y_i$ with $x_i,y_i\in \mathscr{N}_{\varphi}$.\\

\noindent As far as applicable, we use the same notation when considering, more generally, operator valued weights. We also remark here that if $M_1$ and $M_2$ are von Neumann algebras, and $\varphi$ an nsf weight on $M_2$, we can make sense of $(\iota\otimes \varphi)$ as an nsf operator valued weight from $M_1\otimes M_2$ to $M_1$ in a natural way: if $x\in (M_1\otimes M_2)^+$, we let $(\iota\otimes \varphi)(x)$ be the element \[ \omega \in (M_1)_*^+ \rightarrow \varphi((\omega\otimes \iota)(x)) \in \lbrack 0,+\infty\rbrack \] in the extended positive cone of $M_1$. For more information concerning the theory of weights and operator valued weights, we refer to the first chapters of \cite{Tak1}.\\

\section{Compact and von Neumann algebraic quantum groups}

\noindent As we mentioned in the introduction, S.L. Woronowicz developed the notion of a compact matrix quantum group in \cite{Wor1} (there called compact matrix \emph{pseudo}group). Later, he also introduced the more general notion of a compact quantum group (\cite{Wor2}).

\begin{Def}\label{DefComQG} A \emph{compact quantum group} consists of a couple $(A,\Delta)$, where $A$ is a unital C$^*$-algebra, and $\Delta$ a unital $^*$-homomorphism $A\rightarrow A\otimesmin A$ such that \begin{enumerate} \item the map $\Delta$ is coassociative: \[(\Delta\otimes \iota)\Delta = (\iota\otimes \Delta)\Delta,\] and \item the linear subspaces \[\Delta(A)(1\otimes A) := \{\sum_i \Delta(a_i)(1\otimes b_i)\mid a_i,b_i\in A\}\] and \[\Delta(A)(A\otimes 1) := \{\sum_i \Delta(a_i)(b_i\otimes 1)\mid a_i,b_i\in A\}\] are norm-dense in $A\otimesmin A$. \end{enumerate}

\noindent The compact quantum group is called a \emph{compact matrix quantum group} if there exists $n\in \mathbb{N}_0$ and a unitary $u = \sum_{i,j=1}^n u_{ij}\otimes e_{ij} \in A\otimes M_n(\mathbb{C})$, such that $\Delta(u_{ij}) = \sum_{k=1}^n u_{ik}\otimes u_{kj}$ and such that the $u_{ij}$ generate $A$ as a unital C$^*$-algebra. Such a unitary is then called a \emph{fundamental unitary corepresentation}.
\end{Def}

\noindent It is not difficult to show that any compact quantum group $(A,\Delta)$ with $A$ commutative is of the form $(C(\mathfrak{G}),\Delta)$ for some compact group $\mathfrak{G}$, and $\Delta$ dual to the group multiplication. Moreover, $(A,\Delta)$ will then be a compact matrix quantum group iff $\mathfrak{G}$ is a compact Lie group. It is common practice to denote, by analogy, also a general compact quantum group $(A,\Delta)$ as $(C(\mathfrak{G}),\Delta)$, although this notation is now of course purely formal, since there is no underlying object $\mathfrak{G}$.\\

\noindent In \cite{Wor2} (and \cite{VDae1} for the non-separable case), it is proven that to any compact quantum group $(C(\mathfrak{G}),\Delta)$ one can associate a unique state $\varphi$ satisfying \[(\iota\otimes \varphi)\Delta(a) = \varphi(a)1=(\varphi\otimes \iota)\Delta(a),\qquad \forall a\in C(\mathfrak{G}).\] This state is called the \emph{invariant state} of the compact quantum group. It is also proven there that with any compact quantum group, one can associate a Hopf $^*$-algebra $(\textrm{Pol}(\mathfrak{G}),\Delta)$ such that $\textrm{Pol}(\mathfrak{G})\subseteq C(\mathfrak{G})$ is a norm-dense sub-$^*$-algebra, the comultiplication being the restriction of the comultiplication on $C(\mathfrak{G})$. The invariant state is then faithful on $\textrm{Pol}(\mathfrak{G})$. Conversely, any Hopf $^*$-algebra $(\textrm{Pol}(\mathfrak{G}),\Delta)$ possessing an invariant state can be completed to a compact quantum group in essentially two ways. The first construction gives the associated \emph{reduced} compact quantum group. Its underlying C$^*$-algebra $C_r(\mathfrak{G})$ is given as the closure of the image of the GNS-representation of $\textrm{Pol}(\mathfrak{G})$ with respect to the invariant state. The second construction gives the associated \emph{universal} compact quantum group. Its underlying C$^*$-algebra $C_u(\mathfrak{G})$ is now the universal C$^*$-envelope of $\textrm{Pol}(\mathfrak{G})$ (which can be shown to exist). For coamenable compact quantum groups, which are those compact quantum groups possessing both a bounded counit and a \emph{faithful} invariant state, the reduced and universal construction for the underlying Hopf $^*$-algebra both coincide with the original compact quantum group, so that in this case, one only has to specify the Hopf $^*$-algebra to determine completely the associated C$^*$-algebraic structure.\\

\noindent With any compact quantum group $C(\mathfrak{G})$, one can also associate a von Neumann algebra, which we will denote as $\mathscr{L}^{\infty}(\mathfrak{G})$. It is the $\sigma$-weak closure of the image of $\textrm{Pol}(\mathfrak{G})$ under the GNS-representation for $\varphi$. Then $\Delta$ can be completed to a normal unital $^*$-homomorphism $\Delta: \mathscr{L}^{\infty}(\mathfrak{G})\rightarrow \mathscr{L}^{\infty}(\mathfrak{G})\otimes \mathscr{L}^{\infty}(\mathfrak{G})$. This makes $(\mathscr{L}^{\infty}(\mathfrak{G}),\Delta)$ into a von Neumann algebraic quantum group, whose definition we now present.

\begin{Def}\label{DefvNeuQG} (\cite{Kus2}) A \emph{von Neumann algebraic quantum group} is a couple $(M,\Delta)$, consisting of a von Neumann algebra $M$ and a normal unital $^*$-homomorphism $\Delta: M\rightarrow M\otimes M$, such that \begin{enumerate} \item the map $\Delta$ is coassociative: \[(\Delta\otimes \iota)\Delta = (\iota\otimes \Delta)\Delta,\] and \item there exist normal, semi-finite, faithful (nsf) weights $\varphi$ and $\psi$ on $M$ such that, for any state $\omega\in M_*$, we have \[\varphi((\omega\otimes \iota)\Delta(x)) = \varphi(x), \qquad \forall x\in \mathscr{M}_{\varphi}^+\]and\[\psi((\iota\otimes \omega)\Delta(x)) = \psi(x), \qquad \forall x\in \mathscr{M}_{\psi}^+.\] \end{enumerate} \end{Def}

\noindent The weights appearing in this definition turn out to be unique (up to multiplication with a positive non-zero scalar), and are called respectively the \emph{left and right invariant weights}.\\

\noindent The von Neumann algebraic quantum groups $(\mathscr{L}^{\infty}(\mathfrak{G}),\Delta)$ arising from compact quantum groups can be characterized as those von Neumann algebraic quantum groups $(M,\Delta)$ which have a left invariant normal \emph{state}. One can also show that from such a $(\mathscr{L}^{\infty}(\mathfrak{G}),\Delta)$, a $\sigma$-weakly dense Hopf $^*$-subalgebra $(\textrm{Pol}(\mathfrak{G}),\Delta)$ can be reconstructed, providing a one-to-one correspondence between von Neumann algebraic quantum groups with an invariant normal state and Hopf $^*$-algebras with an invariant state.\\

\noindent We now introduce the notion of a unitary 2-cocycle.

\begin{Def} Let $(M,\Delta)$ be a von Neumann algebraic quantum group. A \emph{unitary 2-cocycle} for $(M,\Delta)$ is a unitary element $\Omega\in M\otimes M$ satisfying the 2-cocycle equation \[(\Omega\otimes 1)(\Delta\otimes \iota)(\Omega) = (1\otimes \Omega)(\iota\otimes \Delta)(\Omega).\]\end{Def}

\noindent It is then easily seen that if $(M,\Delta)$ is a von Neumann algebraic quantum group, and $\Omega$ a unitary 2-cocycle for it, we can define a new coproduct $\Delta_{\Omega}$ on $M$ by putting \begin{equation}\label{Def2Coc}\Delta_{\Omega}(x) := \Omega \Delta(x)\Omega^*, \qquad \forall x\in M.\end{equation} The coassociativity of $\Delta_{\Omega}$ then follows precisely from the 2-cocycle equation. A non-trivial result from \cite{DeC2} states that $(M,\Delta_{\Omega})$ in fact possesses invariant nsf weights, so that it is a von Neumann algebraic quantum group. The construction of these weights is rather complicated, and relies on some deep theorems from non-commutative integration theory. However, in the concrete example which we develop in the next section, we will be able to construct the weights on our cocycle twisted quantum group in a fairly straightforward way.\\

\section{Twisting does not preserve compactness}

\noindent The compact quantum groups which we will need will be constructed using the `twisted $SU(2)$' groups from \cite{Wor1} (see also \cite{Wor3}). We recall their definition.

\begin{Def-Prop} Let $q$ be a real number with $0< |q|\leq 1$. Define $\textrm{Pol}(SU_q(2))$ as the unital $^*$-algebra, generated (as a unital $^*$-algebra) by two generators $a$ and $b$ satisfying the relations \[\left\{\begin{array}{lllclll} a^*a +b^*b = 1 && \!\!\!\!\!\!\!\!\!\!\!\!ab = qba \\aa^* + q^2 bb^* = 1 &&\!\!\!\!\!\!\!\!\!\!\!\! a^*b = q^{-1}ba^* \\ &\!\!\!\!\!\!\!\!bb^* = b^*b.\end{array}\right.\]

\noindent Then there exists a Hopf $^*$-algebra structure $(\textrm{Pol}(SU_q(2)),\Delta)$ on $\textrm{Pol}(SU_q(2))$ which satisfies \[\left\{\begin{array}{l} \Delta(a) = a\otimes a - q b^* \otimes b \\ \Delta(b) = b\otimes a + a^*\otimes b.\end{array}\right.\]

\noindent Moreover, this Hopf $^*$-algebra possesses an invariant state $\varphi$, and has a \emph{unique} completion to a compact matrix quantum group $(C(SU_q(2)),\Delta)$.\end{Def-Prop}

\noindent We will always use $a$ and $b$ to denote the generators of $C(SU_q(2))$. Later on however, when we will be working with a sequence of $SU_{q_n}(2)$'s, we will index these generators by the corresponding index $n$ of $q_n$. We will follow the same convention for the comultiplication, the invariant state, and the other special elements in $C(SU_q(2))$ which we will later introduce.\\

\noindent The following proposition gives us more information about how the underlying C$^*$-algebra and the associated invariant state of $C(SU_q(2))$ look like.

\begin{Prop}\label{PropConcRep} (\cite{Wor3}) Let $q$ be a real number with $0< |q| < 1$. Let $\mathscr{H}$ be the Hilbert space $l^2(\mathbb{N}) \otimes l^2(\mathbb{Z})$, whose canonical basis elements we denote as $\xi_{n,k}$ (and with the convention $\xi_{n,k} =0$ when $n<0$). Then there exists a faithful unital $^*$-representation of $C(SU_q(2))$ on $\mathscr{H}$, determined by \[\left\{\begin{array}{l} \pi(a) \xi_{n,k} = \sqrt{1-q^{2n}} \xi_{n-1,k},\\ \pi(b)\xi_{n,k} = q^{n} \xi_{n,k+1}.\end{array}\right.\]

\noindent The invariant state $\varphi$ on $C(SU_q(2))$ is given by \[\varphi(x) = (1-q^2) \sum_{n\in \mathbb{N}} q^{2n} \langle \pi(x) \xi_{n,0},\xi_{n,0}\rangle.\]

\end{Prop}

\noindent We now introduce special elements which will be of importance later on. Define, in the notation of the previous proposition, matrix units $f_{kl}$ on $l^2(\mathbb{N})\otimes l^2(\mathbb{Z})$, by putting $f_{mn} \xi_{r,k} = \delta_{r,n} \xi_{m,k}$, with $\delta$ the Kronecker delta, and where $m,n$ take values in $\mathbb{N}$. It is clear then that each $f_{mn}$ is in the unital C$^*$-algebra generated by $\pi(a)$. Hence $e_{mn} := \pi^{-1}(f_{mn})$ is an element of $C(SU_q(2))$. We define elements $p,p',w \in C(SU_q(2))$ by the following formulas:\begin{eqnarray}\label{Eqnp} p&:=& e_{00}, \\ \label{Eqnp'}  p' &:=& e_{11}, \\ \label{Eqnw} w &:=& e_{01}+e_{12}+ e_{20} + \sum_{k=3}^{\infty} e_{kk}.\end{eqnarray} Thus, with respect to the matrix units $e_{mn}$, the element $w$ is the unitary \[ w = \left(\begin{array}{llll} 0 & 1 & 0 & 0\\ 0 & 0 & 1 & 0 \\ 1 & 0 & 0 & 0 \\ 0 & 0 &0 & I \end{array}\right).\]

\vspace{0.7cm}

\noindent We need to form infinite tensor products of the above quantum groups (see \cite{Wan2} for more detailed information). Let $\mathbf{q}:=(q_n)_{n\in \mathbb{N}_0}$ be a sequence of reals satisfying $0<|q_n|< 1$. Then we can form an inductive sequence \[\overset{n}{\underset{k=1}{\odot}}(\textrm{Pol}(SU_{q_k}(2)),\Delta_{k})\] of Hopf $^*$-algebras, by tensoring with 1 to the right at each step. We denote the inductive limit as $(\textrm{Pol}(SU_{\mathbf{q}}(2)),\Delta_{\mathbf{q}})$. It is easy to see that this is again a Hopf $^*$-algebra with an invariant state $\varphi_{\mathbf{q}}$, which is the pointwise limit of the functionals $\overset{n}{\underset{k=1}{\odot}} \varphi_{k}$. The associated compact quantum group will then again be coamenable, with underlying C$^*$-algebra $C(SU_{\mathbf{q}}(2))$ the universal infinite tensor product of the $C(SU_{q_k}(2))$. This will equal the reduced C$^*$-tensor product with respect to the states $\varphi_{k}$, by nuclearity of the $C(SU_{q_k}(2))$ (see the appendix of \cite{Wor3}).\\

\noindent We will show now that those $(C(SU_{\mathbf{q}}(2)),\Delta_{\mathbf{q}})$ for which $\mathbf{q}$ is square summable possess the property enunciated in the abstract, namely: they possess a unitary 2-cocycle which allows to twist them into a non-compact quantum group. Therefore, \emph{we now fix some $\mathbf{q}$ satisfying the property of square summability.}\\

\noindent We need to show some properties of $\mathscr{L}^{\infty}(SU_{\mathbf{q}}(2))$. We begin with some well-known general lemmas.

\begin{Lem}\label{LemConL2} Let $M\subseteq B(\mathscr{H})$ be a von Neumann algebra, and $\xi\in \mathscr{H}$ a separating vector for $M$. Let $x_n$ be a bounded sequence in $M$ for which $x_n \xi$ converges to a vector $\eta$. Then $x_n$ converges in the $\sigma$-strong topology.\end{Lem}

\begin{Lem}\label{LemConL2Ten} (\cite{Gui1}, \emph{Proposition 1.\i}) Let $(\mathscr{H}_n,\xi_n)$ be a sequence of Hilbert spaces with distinguished unit vectors, and let $(\mathscr{H},\xi)$ be their infinite tensor product. Let $\eta_n \in \mathscr{H}_n$ be non-zero vectors with $\|\eta_n\|\leq 1$. Then $(\overset{n}{\underset{k=1}{\otimes}} \eta_k)\otimes (\overset{\infty}{\underset{k=n+1}{\otimes}} \xi_k) \in \mathscr{H}$ converges to a non-zero vector $\overset{\infty}{\underset{k=1}{\otimes}} \eta_k$ if $\sum | 1-\langle \eta_k,\xi_k\rangle| <\infty$. \end{Lem}

\noindent Recall the special element $p$ (defined at (\ref{Eqnp})), for which we will now also use index notation.

\begin{Lem}\label{LemInfTenp} The sequence $(\overset{n}{\underset{k=1}{\otimes}} p_k)\otimes 1$ converges $\sigma$-strongly in $\mathscr{L}^{\infty}(SU_{\mathbf{q}}(2))$ to an operator $\overset{\infty}{\underset{k=1}{\otimes}} p_k$, while the sequence $1\otimes (\overset{\infty}{\underset{k=n+1}{\otimes}} p_k)$ converges $\sigma$-strongly to 1.\end{Lem}

\begin{proof} Note first that the GNS construction of $\mathscr{L}^{\infty}(SU_{\mathbf{q}}(2))$ with respect to $\varphi_{\mathbf{q}}$ can be identified with $\overset{\infty}{\underset{k=1}{\otimes}} (\mathscr{L}^2(SU_{q_k}(2)),\xi_k)$, with $\xi_k$ the separating and cyclic vector associated to $\varphi_{k}$. Combining Lemma \ref{LemConL2} and Lemma \ref{LemConL2Ten}, we only have to see if $\sum (1-\varphi_{k}(p_k))<\infty$ to have the first statement of the lemma. Since $\varphi_{k}(p_k) = 1-q_k^2$, this follows by the assumption that $q_k$ is a square summable sequence.\\

\noindent Then we can of course also make sense of $x\otimes (\overset{\infty}{\underset{k=n+1}{\otimes}} p_k)$, for any $x\in \overset{n}{\underset{k=1}{\otimes}} \mathscr{L}^{\infty}(SU_{q_k}(2))$. Since \begin{eqnarray*} \| ((\overset{n}{\underset{k=1}{\otimes}}\xi_{q_k})\otimes (\overset{\infty}{\underset{k=n+1}{\otimes}} p_k\xi_{q_k})) - \overset{\infty}{\underset{k=1}{\otimes}}\xi_{q_k} \|^2 &=& 1-\prod_{k=n+1}^{\infty} \varphi_{q_k}(p_k)\\ &=& 1-\prod_{k=n+1}^{\infty} (1-q_k^2),\end{eqnarray*} the second statement follows from the convergence of $\overset{\infty}{\underset{k=1}{\prod}} (1-q_k^2)$ to a non-zero number, which in turn follows easily from the square summability of the $q_n$.

\end{proof}

\begin{Cor}\label{CorDens} Denote $E_n(x) = s_nxs_n$, where $s_n = 1\otimes (\overset{\infty}{\underset{k=n+1}{\otimes}} p_k)$ and $x\in \mathscr{L}^{\infty}(SU_{\mathbf{q}}(2))$. Then $E_n(x)$ converges to $x$ in the $\sigma$-strong topology.\end{Cor}

\noindent Recall the special elements $w$ defined by the formula (\ref{Eqnw}). We again denote this element, when we regard it inside some $C(SU_{q_n}(2))$, as $w_n$.

\begin{Lem} Let $0< |q|< 1$. Let $\xi$ be the cyclic separating vector in the GNS-construction for $\mathscr{L}^{\infty}(SU_q(2))$ w.r.t.~the invariant state $\varphi$. Then \[\| (w^*-a^*)\xi\| \leq 3q^2.\] \end{Lem}

\begin{proof} In the matrix representation introduced just after Proposition \ref{PropConcRep}, it is easy to calculate that \[(w-a)(w^*-a^*) = \left(\begin{array}{cccccc} (1-\sqrt{1-q^2})^2 & 0 & 0 & 0 & 0 & \ldots \\ 0 & (1-\sqrt{1-q^4})^2 & 0 & 0 & 0 & \ldots \\ 0 & 0 & 2-q^6 & -\sqrt{1-q^6} & 0 & \ldots \\ 0 & 0 & -\sqrt{1-q^6} & 2-q^8 & -\sqrt{1-q^8} & \ldots \\ 0 & 0 & 0 & -\sqrt{1-q^8} & 2-q^{10} & \ldots \\ \vdots & \vdots & \vdots & \vdots & \vdots & \ddots \end{array}\right).\]

\noindent Since $1-\sqrt{1-c} \leq c$ for $0\leq c \leq 1$, we have \begin{eqnarray*} && \bigback \varphi((w-a)(w^*-a^*))\\ &=& (1-q^2) ((1-\sqrt{1-q^2})^2 + q^2(1-\sqrt{1-q^4})^2 + q^4(2-q^6) + q^6(2-q^8) + \ldots) \\ &\leq& 2(1-q^2)(q^4+ q^{10}+ q^4 + q^6 + \ldots) \\ &\leq &  9q^4.  \end{eqnarray*} So $\| (w^*-a^*)\xi\| \leq 3q^2$.

\end{proof}

\begin{Theorem}\label{PropConv2coc} The $\sigma$-strong$^*$ limit $\overset{\infty}{\underset{n=1}{\otimes}} ((w_n\otimes w_n)\Delta_{n}(w_n^*))$ exists in $\mathscr{L}^{\infty}(SU_{\mathbf{q}}(2))\otimes \mathscr{L}^{\infty}(SU_{\mathbf{q}}(2))$, and determines a unitary 2-cocycle $\Omega$ for $\mathscr{L}^{\infty}(SU_{\mathbf{q}}(2))$.\end{Theorem}

\begin{proof} Remark that $(\mathscr{L}^{\infty}(SU_{\mathbf{q}}(2))\otimes \mathscr{L}^{\infty}(SU_{\mathbf{q}}(2)), \varphi_{\mathbf{q}}\otimes \varphi_{\mathbf{q}})$ can be identified with \[(\overset{\infty}{\underset{k=1}{\otimes}} (\mathscr{L}^{\infty}(SU_{q_k}(2))\otimes (\mathscr{L}^{\infty}(SU_{q_k}(2)))), \overset{\infty}{\underset{k=1}{\otimes}} (\varphi_{k}\otimes \varphi_{k})).\] Then by the lemmas \ref{LemConL2} and \ref{LemConL2Ten}, it is enough to prove that \[\sum_{n=1}^{\infty} | 1- (\varphi_{n}\otimes \varphi_{n})((w_n\otimes w_n)\Delta_n(w_n^*))| <\infty\] to know that $\overset{n}{\underset{k=1}{\otimes}} ((w_k\otimes w_k)\Delta_{k}(w_k^*))$ converges in the $\sigma$-strong$^*$-topology.\\

\noindent Denoting by $\xi_k$ the GNS vector associated with $\varphi_{k}$, we estimate \begin{eqnarray*} &&\bigback \sum_{n=1}^{\infty} | 1- (\varphi_{n}\otimes \varphi_{n})((w_n\otimes w_n)\Delta_{n}(w_n^*))| \\  &\leq & \sum_{n=1}^{\infty} | 1- (\varphi_{n}\otimes \varphi_{n})((a_n\otimes a_n)\Delta_{n}(a_n^*))| \\ && \,\,+ \sum_{n=1}^{\infty} |(\varphi_{n}\otimes \varphi_{n})((a_n\otimes a_n)(\Delta_{n}(a_n^*)-\Delta_{n}(w_n^*)))| \\ &&\,\, + \sum_{n=1}^{\infty} |(\varphi_{n}\otimes \varphi_{n})((a_n\otimes a_n-w_n\otimes w_n)\Delta_{n}(w_n^*))|\\ &\leq & \sum_{n=1}^{\infty} | 1- (\varphi_{n}\otimes \varphi_{n})((a_n\otimes a_n)\Delta_{n}(a_n^*))| \\ &&\,\,+ \sum_{n=1}^{\infty} |(\varphi_{n}\otimes \varphi_{n})(\Delta_{n}((a_n-w_n)(a_n-w_n)^*))|^{1/2} \\ && \,\,+ \sum_{n=1}^{\infty} |(\varphi_{n}\otimes \varphi_{n})((a_n\otimes a_n-w_n\otimes w_n)(a_n\otimes a_n-w_n\otimes w_n)^*)|^{1/2} \\ &= & \sum_{n=1}^{\infty} | 1- (\varphi_{n}\otimes \varphi_{n})((a_n\otimes a_n)\Delta_{n}(a_n^*))| \\ && \, \,+ \sum_{n=1}^{\infty} \|(a_n-w_n)^*\xi_n\| + \sum_{n=1}^{\infty} \|(a_n\otimes a_n-w_n\otimes w_n)^* (\xi_n\otimes \xi_n)\| \\ &\leq& \sum_{n=1}^{\infty} | 1- (\varphi_{n}\otimes \varphi_{n})((a_n\otimes a_n)\Delta_{n}(a_n^*))| + 3 \sum_{n=1}^{\infty} \|(a_n-w_n)^*\xi_n\|.\end{eqnarray*} By the previous lemma and the square summability of the $q_n$, we have \[\sum_{n=1}^{\infty} \|(a_n-w_n)^*\xi_n\|<\infty.\]

\noindent So we only have to compute if \[\sum_{n=1}^{\infty} | 1- (\varphi_{n}\otimes \varphi_{n})((a_n\otimes a_n)\Delta_{n}(a_n^*))|<\infty.\]

\noindent Now an easy calculation shows that \[(\varphi_{n}\otimes \varphi_{n})((a_n\otimes a_n)\Delta_{n}(a_n^*)) = \frac{1}{(1+q_n^2)^2}.\] Since $1-(\frac{1-q_n^2}{1+q_n^2})^2 \leq 2 q_n^2$, we can again conclude convergence by square summability of the $q_n$.\\

\noindent Thus \[\Omega = \overset{\infty}{\underset{n=1}{\otimes}} ((w_n\otimes w_n)\Delta_{n}(w_n^*))\] is a well-defined unitary as a $\sigma$-strong$^*$ limit of unitaries. Since multiplication is jointly continuous on the group of unitaries with the $\sigma$-strong$^*$ topology, $\Omega$ will satisfy the 2-cocycle identity since each $\overset{n}{\underset{k=1}{\otimes}} ((w_k\otimes w_k)\Delta_{k}(w_k^*))$ does.

\end{proof}

\begin{Lem}\label{LemEstw} Let $0<|q| < 1$, and take $w\in C(SU_q(2))$ as defined by the formula (\ref{Eqnw}). Then the invariant state $\varphi$ for $C(SU_q(2))$ satisfies \[ \varphi \leq q^{-2}\varphi(w^*\,\cdot\, w).\]\end{Lem}

\begin{proof} This follows from a straightforward computation, using the concrete form of $\varphi$ as in Proposition \ref{PropConcRep}.

\end{proof}

\noindent Hence we can form on $\mathscr{L}^{\infty}(SU_{\mathbf{q}}(2))$ the normal faithful weight \[\varphi_{\Omega} := \lim_{n\rightarrow \infty} (\prod_{k=1}^n q_k^{-2}) ((\overset{n}{\underset{k=1}{\otimes}}\varphi_{k}(w_k^* \,\cdot \, w_k))\otimes (\overset{\infty}{\underset{k=n+1}{\otimes}}\varphi_{k})),\] the limit being taken pointwise on elements of $\mathscr{L}^{\infty}(SU_{\mathbf{q}}(2))^+$. This is a well-defined normal faithful weight, since it is an increasing sequence of normal, faithful, positive functionals. It is clear that $\varphi_{\Omega}$ is not finite.

\begin{Prop}\label{LemSemFin} The weight $\varphi_{\Omega}$ on $\mathscr{L}^{\infty}(SU_{\mathbf{q}}(2))$ is semi-finite.\end{Prop}

\begin{proof} We prove that the projections $1\otimes (\overset{\infty}{\underset{k=n+1}{\otimes}} p_k)$, introduced in Lemma \ref{LemInfTenp}, are integrable with respect to $\varphi_{\Omega}$. By Corollary \ref{CorDens}, this will prove the proposition.\\

\noindent But $w_n^* p_n w_n = p'_n$, using also the notation of (\ref{Eqnp'}). Since $\varphi_{n}(p'_n) = q_n^2(1-q_n^2)$, the integrability of all $1\otimes (\overset{\infty}{\underset{k=n+1}{\otimes}} p_k)$ follows.

\end{proof}

\noindent We now end by proving that $\varphi_{\Omega}$ is an invariant nsf weight for the couple $(\mathscr{L}^{\infty}(SU_{\mathbf{q}}(2)),\Delta_{\Omega})$, where $\Omega$ was introduced in Theorem \ref{PropConv2coc} and $\varphi_{\Omega}$ just before Lemma \ref{LemSemFin}, and where $\Delta_{\Omega}$ is the twisted coproduct defined by (\ref{Def2Coc}).

\begin{Theorem} The nsf weight $\varphi_{\Omega}$ is a left and right invariant nsf weight for $(\mathscr{L}^{\infty}(SU_{\mathbf{q}}(2)),\Delta_{\Omega})$.\end{Theorem}

\begin{proof} We only prove left invariance, since the proof for right invariance follows by symmetry.\\

\noindent Denote \[\mathscr{P}_n = \{ x\otimes (\overset{\infty}{\underset{k=n+1}{\otimes}} p_k) \mid x\in \overset{n}{\underset{k=1}{\otimes}} \mathscr{L}^{\infty}(SU_{q_k}(2))\},\] still using the notation (\ref{Eqnp}), and denote $\mathscr{P} = \bigcup_{n\in \mathbb{N}_0} \mathscr{P}_n $. By the proof of Lemma \ref{LemSemFin}, we know that $\mathscr{P}$ consists of integrable elements for $\varphi_{\Omega}$. We first show that also $\Delta_{\Omega}(\mathscr{P})\subseteq \mathscr{M}_{(\iota\otimes \varphi_{\Omega})}$, and that \[(\iota \otimes \varphi_{\Omega})(\Delta_{\Omega}(y)) = \varphi_{\Omega}(y)1\] for $y$ in $\mathscr{P}$. \\

\noindent Choose $x= \otimesud{k=1}{n} x_n \in \overset{n}{\underset{k=1}{\otimes}} \mathscr{L}^{\infty}(SU_{q_k}(2))$ with all $x_n$ positive, and put $y = x\otimes (\otimesud{k=n+1}{\infty} p_k)$. Then $(\iota\otimes \varphi_{\Omega})(\Delta_{\Omega}(y))$ is an element in the extended positive cone of $\mathscr{L}^{\infty}(SU_{\mathbf{q}}(2))$. As such, it is the pointwise supremum of the elements \[ z_m = (\otimesud{k=1}{n} \alpha_k) \otimes (\otimesud{k=n+1}{n+m}\beta_k) \otimes (\otimesud{k=n+m+1}{\infty}\gamma_k),\] regarded as semi-linear functionals on $\mathscr{L}^{\infty}(SU_{\mathbf{q}}(2))_*^+$, where \[\alpha_k = q_k^{-2}(\iota\otimes \varphi_{k}(w_k^*\,\cdot\,w_k))((w_k\otimes w_k)\Delta_{k}(w_k^*x_kw_k)(w_k^*\otimes w_k^*)),\] \[\beta_k = q_k^{-2}(\iota\otimes \varphi_{k}(w_k^*\,\cdot\,w_k))((w_k\otimes w_k)\Delta_{k}(w_k^*p_kw_k)(w_k^*\otimes w_k^*))\] and \[\gamma_k = (\iota\otimes \varphi_{k})((w_k\otimes w_k)\Delta_{k}(w_k^*p_kw_k)(w_k^*\otimes w_k^*)),\] and where the infinite tensor product is to be seen as a $\sigma$-strong limit.\\

\noindent We can simplify $\alpha_k$ and $\beta_k$ to respectively $q_k^{-2}\varphi_{k}(w_k^*x_kw_k)1$ and $q_k^{-2}\varphi_{k}(w_k^*p_kw_k)1$, by invariance of $\varphi_{k}$, while $\gamma_k$ satisfies \[\gamma_k \leq q_k^{-2}\varphi_{k}(w_k^*p_kw_k)1\] by Lemma \ref{LemEstw} and left invariance of $\varphi_{k}$. This implies that $z_m \leq \varphi_{\Omega}(y)1$. Since $z_m$ is increasing, and the invariant state $\varphi_{\mathbf{q}}$ on $\mathscr{L}^{\infty}(SU_{\mathbf{q}}(2))$ is a faithful normal state, we only have to prove that $\varphi_{\mathbf{q}}(z_m) \rightarrow \varphi_{\Omega}(y)$ to conclude that $(\iota\otimes \varphi_{\Omega})(\Delta_{\Omega}(y)) = \varphi_{\Omega}(y)1$.\\

\noindent But it is easily seen that $\varphi_{\mathbf{q}}(z_m)$ converges to $\varphi_{\Omega}(y)$ if we can show that \[\prod_{k=n+m+1}^{\infty}\varphi_{k}(\gamma_k) \underset{m\rightarrow \infty}{\rightarrow} 1.\] This is equivalent with proving that \[\prod_{1}^{\infty} (\varphi_{k}\otimes \varphi_{k})((w_k\otimes w_k)\Delta_{k}(w_k^*p_kw_k)(w_k^*\otimes w_k^*)) \neq 0.\] Since the left hand side equals $(\varphi_{\mathbf{q}}\otimes \varphi_{\mathbf{q}})(\Omega\Delta_{\mathbf{q}}(s)\Omega^*)$, with $s = \otimesud{k=1}{\infty} p_k$, the above product is indeed non-zero, by faithfulness of $\varphi_{\mathbf{q}}$.\\

\noindent One then easily concludes that, since any $y\in \mathscr{P}$ is a linear combination of elements of the above form, we have $\Delta_{\Omega}(y) \in\mathscr{M}_{\iota\otimes \varphi_{\Omega}}$ for $y\in \mathscr{P}$, with $(\iota\otimes \varphi_{\Omega})\Delta_{\Omega}(y) = \varphi_{\Omega}(y)$.\\

\noindent Next we prove that $\mathscr{P}$ is a $\sigma$-strong-norm core for the GNS-map $\Lambda_{\varphi_{\Omega}}$ associated with the nsf weight $\varphi_{\Omega}$. For this, it is enough to prove that $s_n = 1\otimes (\overset{\infty}{\underset{k=n+1}{\otimes}} p_k)$ is invariant under $\sigma_t^{\varphi_{\Omega}}$ for any $t\in \mathbb{R}$, since then, using Corollary \ref{CorDens}, we can conclude that, whenever $y\in \mathscr{N}_{\varphi_{\Omega}}$, we will have $s_nys_n\rightarrow y$ in the $\sigma$-strong topology and $\Lambda_{\varphi_{\Omega}}(s_nys_n) = s_n J_{\varphi_{\Omega}} s_nJ_{\varphi_{\Omega}} \Lambda_{\varphi_{\Omega}}(y) \rightarrow \Lambda_{\varphi_{\Omega}}(y)$ in the norm topology (where $J_{\varphi_{\Omega}}$ is the modular conjugation on $\mathscr{L}^2(M,\varphi_{\Omega})$).\\

\noindent By formula (A 1.4) of \cite{Wor1} (or by a direct verification), we have that $\sigma_t^{\varphi_k}(a_k) =|q|^{2it}a_k$ for any $t\in \mathbb{R}$. Hence $a_ka_k^*$ is in the centralizer of $\varphi_k$, and so the same is true of $p_k$ and $p_k'$. Since $s_n\in \mathscr{M}_{\varphi_\Omega}$, we will have $xs_n, s_nx \in \mathscr{M}_{\varphi_{\Omega}}$ for $x\in \mathscr{N}_{\varphi_{\Omega}}\cap \mathscr{N}_{\varphi_{\Omega}}^*$, and then \begin{eqnarray*} \varphi_{\Omega}(s_nx) &=& \lim_{m\rightarrow \infty} (\prod_{k=1}^{n+m} q_k^{-2}) ((\overset{n}{\underset{k=1}{\otimes}}\varphi_{k}(w_k^* \,\cdot \, w_k))\otimes (\overset{n+m}{\underset{k=n+1}{\otimes}}\varphi_{k}(w_k^*p_k \,\cdot \, w_k))\otimes  (\overset{\infty}{\underset{k=n+m+1}{\otimes}}\varphi_{k}(p_k \,\cdot \,)))(x) \\ &=& \lim_{m\rightarrow \infty} (\prod_{k=1}^{n+m} q_k^{-2}) ((\overset{n}{\underset{k=1}{\otimes}}\varphi_{k}(w_k^* \,\cdot \, w_k))\otimes (\overset{n+m}{\underset{k=n+1}{\otimes}}\varphi_{k}(p_k'w_k^* \,\cdot \, w_k))\otimes  (\overset{\infty}{\underset{k=n+m+1}{\otimes}}\varphi_{k}(p_k \,\cdot \,)))(x) \\ &=& \lim_{m\rightarrow \infty} (\prod_{k=1}^{n+m} q_k^{-2}) ((\overset{n}{\underset{k=1}{\otimes}}\varphi_{k}(w_k^* \,\cdot \, w_k))\otimes (\overset{n+m}{\underset{k=n+1}{\otimes}}\varphi_{k}(w_k^* \,\cdot \, w_kp_k'))\otimes  (\overset{\infty}{\underset{k=n+m+1}{\otimes}}\varphi_{k}( \,\cdot \,p_k)))(x) \\ &=& \varphi_{\Omega}(xs_n).\end{eqnarray*} From this, the equality $\sigma_t^{\varphi_{\Omega}}(s_n)=s_n$ for any $t\in \mathbb{R}$ follows (for example by Theorem VIII.2.6 of \cite{Tak1}).\\

\noindent We can now conclude the proof. By left-invariance of $\varphi_{\Omega}$ on $\mathscr{P}$, we can introduce an isometry $W_{\Omega}^*$ on $\mathscr{L}^2(SU_{\mathbf{q}}(2),\varphi_{\Omega}) \otimes \mathscr{L}^2(SU_{\mathbf{q}}(2),\varphi_{\Omega})$ by putting \[W_{\Omega}^* (\Lambda_{\varphi_{\Omega}}(x)\otimes \Lambda_{\varphi_{\Omega}}(y)) = (\Lambda_{\varphi_{\Omega}}\otimes \Lambda_{\varphi_{\Omega}})(\Delta_{\Omega}(y)(x\otimes 1))\] for $x\in \mathscr{N}_{\varphi_{\Omega}}$ and $y\in \mathscr{P}$. Then by the core-property of $\mathscr{P}$, we conclude that for  $x\in \mathscr{N}_{\varphi_{\Omega}}$ and \emph{any} $y\in \mathscr{N}_{\varphi_{\Omega}}$, we have $\Delta_{\Omega}(y)(x\otimes 1)$ square integrable for $\varphi_{\Omega}\otimes \varphi_{\Omega}$, with \[W_{\Omega}^* (\Lambda_{\varphi_{\Omega}}(x)\otimes \Lambda_{\varphi_{\Omega}}(y)) = (\Lambda_{\varphi_{\Omega}}\otimes \Lambda_{\varphi_{\Omega}})(\Delta_{\Omega}(y)(x\otimes 1)).\] Hence \[\varphi_{\Omega}((\omega \otimes \iota)\Delta_{\Omega}(y)) = \varphi_{\Omega}(y)\] for $y\in \mathscr{M}_{\varphi_{\Omega}}^+$ and $\omega$ a state of the form $\langle \, \cdot \, \Lambda_{\varphi_{\Omega}}(x),\Lambda_{\varphi_{\Omega}}(x)\rangle$ with $x\in \mathscr{N}_{\varphi_{\Omega}}$. Hence the elements $(\iota\otimes \varphi_{\Omega})(\Delta_{\Omega}(y))$ and $\varphi_{\Omega}(y)1$ in the extended cone of $M$ are equal on a normdense subset of $M_*^+$. Since the latter element is bounded, the same is true of the former by lower-semi-continuity, and then their equality everywhere follows.\\

\end{proof}

\noindent Since the C$^*$-algebra underlying a non-compact von Neumann algebraic quantum group is non-unital, we obtain the following corollary.

\begin{Cor}\label{CorTop} There exists a von Neumann algebraic quantum group $(M,\Delta)$ and a unitary 2-cocycle $\Omega\in M\otimes M$, such that the reduced (resp.~universal) C$^*$-algebra associated to $(M,\Delta)$ is \emph{not} isomorphic to the reduced (resp.~universal) C$^*$-algebra associated to $(M,\Delta_{\Omega})$, the $\Omega$-twisted von Neumann algebraic quantum group.\end{Cor}

\noindent \emph{Remarks:} 1. Note that the compact quantum group $\mathscr{L}^{\infty}(SU_{\mathbf{q}}(2))$ is \emph{not} a compact matrix quantum group. It would therefore be interesting to see if one can also twist compact \emph{matrix} quantum groups into non-compact locally compact quantum groups. By the results of \cite{DeC3}, this is closely related to the question whether a compact matrix quantum group can act ergodically on an infinite-dimensional type $I$-factor.\\

\hspace{1.1cm} 2. It follows from the results of \cite{DeC1} that if $(C(\mathfrak{G}),\Delta)$ is a compact quantum group, and $\Omega$ a unitary 2-cocycle \emph{inside} $\textrm{Pol}(\mathfrak{G})\odot \textrm{Pol}(\mathfrak{G})$, then the cocycle twisted von Neumann algebraic quantum group is again compact, with $\textrm{Pol}(\mathfrak{G})$ as the $^*$-algebra underlying the associated Hopf $^*$-algebra. Hence also the associated C$^*$-algebras remain unaltered.\\

\noindent \emph{Acknowledgements:} I would like to thank my thesis advisor Alfons Van Daele for his unwavering support of my work, and Stefaan Vaes for valuable suggestions concerning the presentation of the results in this article.

\end{document}